\documentclass[11pt,a4paper]{amsart}
\usepackage[utf8]{inputenc}
\usepackage[english]{babel}
\usepackage[left=2.5cm,right=2.5cm,top=2.5cm,bottom=2.5cm]{geometry}
\usepackage{amsmath,amssymb,latexsym, amsthm}
\usepackage{graphicx}

\newcommand{\inter}{\mathrm{int}}
\newcommand{\conv}{\mathrm{conv}}
\newcommand{\per}{\mathrm{per}}
\newcommand{\area}{\mathrm{area}}

\newcommand{\vol}{\mathrm{vol}}

\newtheorem{thm}{Theorem}[section]
\newtheorem{lem}[thm]{Lemma}
\newtheorem{cor}[thm]{Corollary}

\newtheorem*{KPC}{Kneser--Poulsen Conjecture}
\newtheorem*{DKPC}{Dual Kneser--Poulsen Conjecture}
\newtheorem*{claim}{Claim}
\theoremstyle{definition}

\newtheorem{defin}[thm]{Definition}

\makeatletter
\DeclareFontFamily{U}{tipa}{}
\DeclareFontShape{U}{tipa}{m}{n}{<->tipa10}{}
\newcommand{\arc@char}{{\usefont{U}{tipa}{m}{n}\symbol{62}}}%

\newcommand{\arc}[1]{\mathpalette\arc@arc{#1}}

\newcommand{\arc@arc}[2]{%
  \sbox0{$\m@th#1#2$}%
  \vbox{
    \hbox{\resizebox{\wd0}{\height}{\arc@char}}
    \nointerlineskip
    \box0
  }%
}
\makeatother

\title[Two Kneser--Poulsen-type Inequalities]{Two Kneser--Poulsen-type Inequalities in Planes of Constant Curvature}
\begin{document}
\author[B. Csik\'os]{Bal\'azs Csik\'os}
\address{B. Csik\'os, Dept.\ of Geometry, E\"otv\"os Lor\'and University, Budapest, P\'azm\'any P. stny. 1/C, Hungary, H-1117}
\email{csikos@cs.elte.hu}

\author[M. Horv\'ath]{M\'arton Horv\'ath}
\address{M. Horv\'ath, Dept.\ of Geometry, Budapest University of Technology and Economics, Budapest, Egry J\'ozsef u. 1, Hungary, H-1111} 
\email{horvathm@math.bme.hu}
\thanks{The research was supported by the National Research Development and Innovation Office (NKFIH) Grant No. OTKA K112703.}
\date{}
\dedicatory{Dedicated to Ted Bisztriczky, Gábor Fejes Tóth, and Endre Makai \break on the occasion of their 70th birthdays.}

\begin{abstract}
We show that the perimeter of the convex hull of finitely many disks lying in the hyperbolic or Euclidean plane, or in a hemisphere does not increase when the disks are rearranged so that the distances between their centers do not increase. This generalizes the theorem on the monotonicity of the perimeter of the convex hull of a finite set under contractions, proved in the Euclidean plane by V.~N.~Sudakov \cite{Sudakov}, R.~Alexander \cite{Alexander}, V.~Capoyleas and J.~Pach \cite{Pach_Capoyleas}.  We also prove that the area of the intersection of finitely many disks in the hyperbolic plane does not decrease after such a contractive rearrangement. The Euclidean analogue of the latter statement was proved by K.~Bezdek and R.~Connelly \cite{Bezdek_Connelly}. Both theorems are proved by a suitable adaptation of a recently published method of I.~Gorbovickis \cite{Gorbovickis}.
\end{abstract}

\maketitle

\section{Introduction}
A map $f\colon X\to Y$ between two metric spaces $(X,d_X)$ and $(Y,d_Y)$ satisfying the inequality $d_X(p,q)\geq d_Y(f(p),f(q))$ for any pair of points $p,q\in X$ is usually called a \emph{weak contraction}, but we shall call such a map simply a \emph{contraction}. We shall also say that a system of points $(p_1',\dots,p_k')\in Y^k$ is a \emph{contraction of the system of points} $(p_1,\dots,p_k) \in X^k$ if $d_X(p_i,p_j)\geq d_Y(p_i',p_j')$ for all pairs of indices $1\leq i,j\leq k$. 

There are many theorems and conjectures saying that a set is not \emph{smaller} than its contractions if \emph{largeness} is measured by a certain geometric quantity. 

For example, it is clear that the diameter and, for any $s\geq 0$, the $s$-dimensional outer Hausdorff measure of a metric space can not increase when the space is contracted. The question whether the outer Minkowski contents of a compact subset of the Euclidean space $\mathbb E^n$ can increase when the set is contracted led M.~Kneser \cite{Kneser} to the conjecture, proposed also by E.~T.~Poulsen \cite{Poulsen}, that the volume of the $\varepsilon$-neighborhood of a finite set in the Euclidean space $\mathbb E^n$ is not smaller than that of any of its contractions in $\mathbb E^n$ for any $\varepsilon>0$. Based on the known special cases of this conjecture, one can put forward a more general conjecture which we shall also refer to as the 
\begin{KPC} If $M$ is either the hyperbolic space $\mathbb H^n$, or the  Euclidean space $\mathbb E^n$, or the sphere $\mathbb S^n$, and if the system of points $(p_1',\dots,p_k')\in M^k$ is a contraction of the system of points $(p_1,\dots,p_k)\in M^k$, then for any choice of the radii $(r_1,\dots,r_k)$, we have
	\[
	\vol_n\left(\bigcup_{i=1}^kB(p_i,r_i)\right)\geq \vol_n\left(\bigcup_{i=1}^kB(p_i',r_i)\right),
	\]
where $\vol_n$ is the $n$-dimensional volume function in $M$, and $B(p,r)$ denotes the closed ball of radius $r$ in $M$ centered at $p$.
\end{KPC}
The following consequence of the Euclidean Kneser--Poulsen conjecture was proved by V.~N.~Sudakov \cite{Sudakov}, R.~Alexander \cite{Alexander}, V.~Capoyleas and J.~Pach \cite{Pach_Capoyleas}.

\begin{thm} \label{mean_width} The mean width of the convex hull of a finite subset of $\mathbb E^n$ is not less than the mean width of the convex hull of any of its contractions in $\mathbb E^n$. 
\end{thm}
The perimeter of a compact convex set in $\mathbb E^2$ is just $\pi$ times its mean width, so Theorem \ref{mean_width} can be rephrased in the 2-dimensional case as follows.
\begin{cor}\label{per_conv}	If $X$ is a finite subset of $\mathbb E^2$, and $X'\subset \mathbb E^2$ is a contraction of $X$, then
	\[\per\left(\conv\left(X\right)\right)\geq\per\left(\conv\left(X'\right)\right),\]
where $\per(A)$ and $\conv(A)$ denote the perimeter and the convex hull of the set $A\subset \mathbb E^2$, respectively.
\end{cor}

We can also formulate the  
\begin{DKPC} If $M\in \{ \mathbb H^n, \mathbb E^n, \mathbb S^n\}$, and if the configuration $(p_1',\dots,p_k')\allowbreak\in M^k$ is a contraction of the configuration $(p_1,\dots,p_k)\in M^k$, then for any choice of the radii $(r_1,\dots,r_k)$, we have
	\[
	\vol_n\left(\bigcap_{i=1}^kB(p_i,r_i)\right)\leq \vol_n\left(\bigcap_{i=1}^kB(p_i',r_i)\right).
	\]
\end{DKPC}
For the sphere $\mathbb S^n$, the Kneser--Poulsen conjecture is equivalent to its dual, since the complement of a closed ball in $\mathbb S^n$ is an open ball, but the two conjectures seem to be independent in the hyperbolic and Euclidean cases. The dual Kneser--Poulsen conjecture can also be thought of as a quantitative version of the Kirszbraun--Valentine theorem. 

\begin{thm}[M.~Kirszbraun, F.~A.~Valentine]\label{Kirszbraun}
If $M\in\{\mathbb H^n, \mathbb E^n,\mathbb S^n\}$, and $X\subseteq M$ is a subset of $M$, $f\colon X\to M$ is a contraction, then there exists a contraction $\tilde f\colon M\to M$ such that $\tilde f|_X=f$.
\end{thm}

Theorem \ref{Kirszbraun} was proved by M.~Kirszbraun \cite{Kirszbraun} for the Euclidean space and extended to the hyperbolic and spherical spaces by F.~A.~Valentine \cite{Valentine_hyperbolic}, \cite{Valentine_sphere}. By a compactness argument, the proof reduces to the statement that if the configuration $(p_1',\dots,p_k')\in M^k$ is a contraction of the configuration $(p_1,\dots,p_k)\in M^k$, and the intersection $\bigcap_{i=1}^k B(p_i,r_i)$ is not empty, then the intersection $\bigcap_{i=1}^k B(p_i',r_i)$ is also not empty.

Both the Kneser--Poulsen conjecture and its dual are completely proved in $\mathbb E^2$ by K.~Bezdek and R.~Connelly \cite{Bezdek_Connelly}, but all the remaining cases are still open. The cited proofs of Theorem \ref{mean_width} and Corollary \ref{per_conv} also use heavily the Euclidean structure of the ambient space and cannot be transferred to $\mathbb H^n$ or $\mathbb S^n$ in an obvious way.

In a recent paper, I.~Gorbovickis \cite{Gorbovickis} introduced a new method, with the help of which he could extend the result of K.~Bezdek and R.~Connelly to the hyperbolic and spherical planes under some extra assumptions. Namely, he proved the following theorems.

\begin{thm}[I.~Gorbovickis]
	 If $M\in\{\mathbb H^2,\mathbb S^2\}$, and if the configuration $(p_1',\dots,p_k')\in M^k$ is a contraction of the configuration $(p_1,\dots,p_k)\in M^k$, then for any choice of the radii $(r_1,\dots,r_k)$ for which the union $\bigcup_{i=1}^kB(p_i,r_i)$ is simply connected, we have
	 \[
	 \area\left(\bigcup_{i=1}^kB(p_i,r_i)\right)\geq \area\left(\bigcup_{i=1}^kB(p_i',r_i)\right).
	 \]
\end{thm}
\begin{thm}[I.~Gorbovickis]\label{thm:Gorbovickis_metszet}
	If the configuration $(p_1',\dots,p_k')\in (\mathbb S^2)^k$ is a contraction of the configuration $(p_1,\dots,p_k)\in (\mathbb S^2)^k$, then for any choice of the radii $(r_1,\dots,r_k)$ for which the intersection $\bigcap_{i=1}^kB(p_i,r_i)$ is  connected, we have
	\[
	\area\left(\bigcap_{i=1}^kB(p_i,r_i)\right)\leq \area\left(\bigcap_{i=1}^kB(p_i',r_i)\right).
	\]
\end{thm}

The main goals of the paper are the following. In Section \ref{sec:2}, we generalize Corollary \ref{per_conv} to point sets lying in the hyperbolic plane or on a hemisphere. Following the ideas of I.~Gorbovickis, we shall use the central set for the proof. To make the inductive proof work, we prove  a stronger statement, claiming that the perimeter of the convex hull of some disks cannot increase when the disks are rearranged by a contraction of the centers. This theorem is new also in the Euclidean plane.

In Section \ref{sec:3}, we prove the dual Kneser--Poulsen conjecture in the hyperbolic plane. The proof of this theorem also pursues the scheme of I.~Gorbovickis, but requires the introduction of a dual version of the central set, which we shall call the co-central set.


\section{The perimeter of the convex hull of finitely many disks \label{sec:2}}

In this section, we prove the following 

\begin{thm}\label{thm:perconv}
If $M\in \{\mathbb H^2,\mathbb E^2, \mathbb S^2\}$,  $B(p_1,r_1),\dots,B(p_k,r_k)$ are closed disks in  $M$ which are contained in a hemisphere in the case of $M=\mathbb S^2$, and $(p_1',\dots,p_k')\in M^k$ is a contraction of $(p_1,\dots,p_k)$, then
\begin{equation}\label{ineq_perconv}
\per\left(\conv\left(\bigcup_{i=1}^kB(p_i,r_i)\right)\right)\geq\per\left(\conv\left(\bigcup_{i=1}^kB(p_i',r_i)\right)\right).
\end{equation}
\end{thm}

The radii are allowed to be equal to $0$. In such a case, the disk $B(p,0)$ is just the singleton $\{p\}$. 
We shall denote the disks $B(p_i,r_i)$ and $B(p_i',r_i)$ by $D_i$ and $D_i'$, respectively. The distance function of $M$ will be denoted by $d$.

\begin{lem}\label{lemma:k_2} Theorem \ref{thm:perconv} is true for $k=2$.
\end{lem} 
\begin{proof}We may assume without loss of generality that $r_1\geq r_2$.
Moving the contracted system by a suitable isometry, it can also be assumed that  $p_1'=p_1$ and $p_2'$ is in the segment $[p_1,p_2]$. If $p_2'=p_2$ or $D_1'\supseteq D_2'$, then inequality \eqref{ineq_perconv} is obviously true. Suppose that $d(p_1,p_2)>d(p_1',p_2')$ and $D_1'\not\supseteq D_2'$. Then the convex hulls of  $D_1\cup D_2$ and  $D_1'\cup D_2'$ are bounded by two segments of the common tangent lines of the disks and two circle arcs.
	
	We are going to use the following fact.
		\begin{figure}[b]
			\centering{\includegraphics[height=4.7cm]{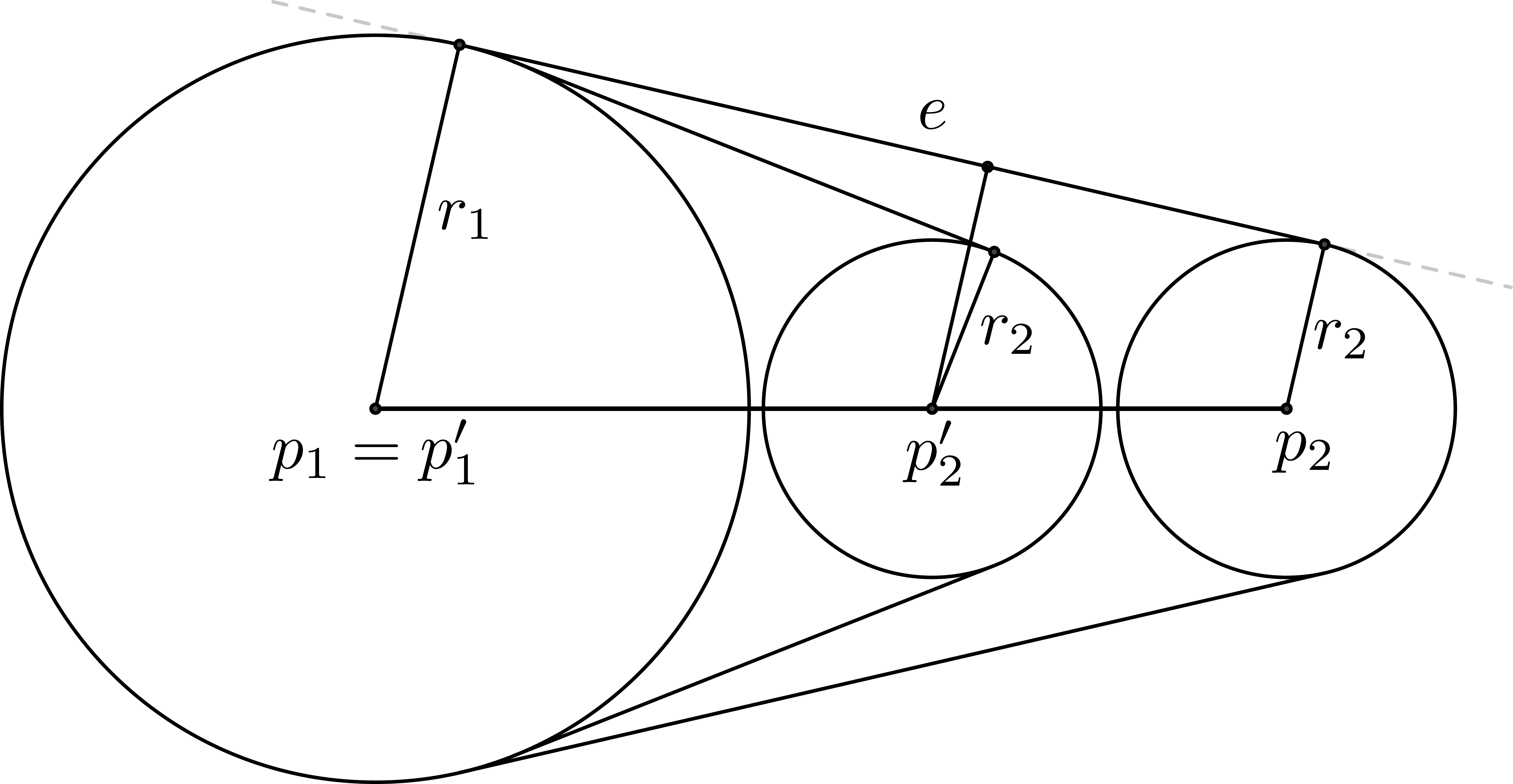}}
			\caption{The case of $2$ disks for $M\in \{\mathbb E^2, \mathbb S^2\}$.\label{k_2_1} }
		\end{figure}
	\begin{claim}	
Let $\gamma\colon [a,b]\to M$ be a parameterized geodesic segment in $M$, disjoint from the complete geodesic line $e$, and consider the distance function $\delta\colon [a,b]\to \mathbb R$, $\delta(t)=d(\gamma(t),e)$. Then $\delta$ is convex if $M=\mathbb H^2$, linear if $M=\mathbb E^2$, and concave if $M=\mathbb S^2$. In particular, if $M\in\{\mathbb H^2,\mathbb E^2\}$, then $\delta$ attains its maximum at one of the endpoints of the interval $[a,b]$; if $M\in\{\mathbb E^2,\mathbb S^2\}$, then $\delta$ attains its minimum either at  $a$ or at $b$.  
	\end{claim}
	
 Applying the Claim for the segment $[p_1,p_2]$ and one of the common tangents $e$ of the disks $D_1$ and $D_2$, we obtain that if $M\in\{\mathbb E^2,\mathbb S^2\}$, then the distance of $p_2'$ from $e$ is at least $\min\{d(p_1,e),d(p_2,e)\}=r_2$, see Figure \ref{k_2_1}. Thus, the disk $D_2'$ is contained in the convex hull $\conv(D_1\cup D_2)$, consequently, we also have
	\[
	\conv(D_1'\cup D_2')\subseteq\conv(D_1\cup D_2).
	\] 
	This containment implies inequality \eqref{ineq_perconv} between the perimeters of the convex hulls.
	
	\begin{figure}[t]
		\centering{\includegraphics[height=6.5cm]{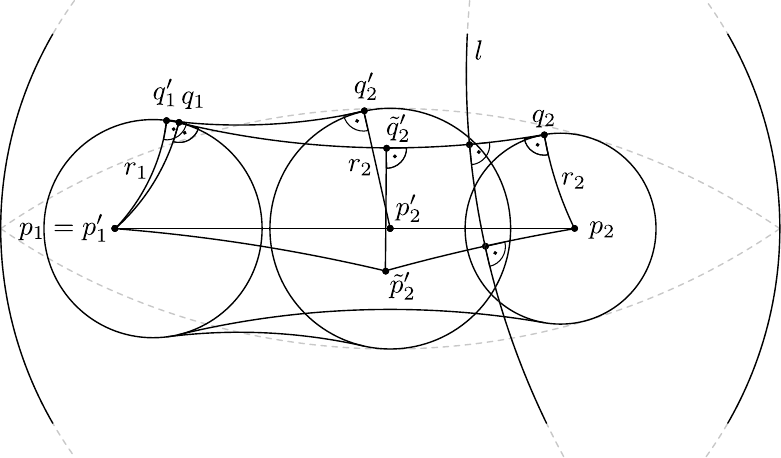}}
		\caption{The case of $2$ disks in $\mathbb H^2$.\label{k_2_2a} }
	\end{figure}
	
	If $M=\mathbb H^2$, then a similar argument completes the proof if the convex hull $\conv(D_1\cup D_2)$ covers the disk $D_2'$, so assume that this is not the case.  An arrangement of disks for which 
	\begin{equation}\label{eq:nonincl}
	D_2'\not \subseteq \conv(D_1\cup D_2),
	\end{equation} constructed in Poincar\'e's disk model of the hyperbolic plane is shown in Figure \ref{k_2_2a}.  The dashed grey lines are the equidistant lines to the straight line $p_1p_2$ with distance $r_2$. They allow us to compare the radii of the disks. 
	
	Choose one of the outer common tangents of the disks $D_1$ and $D_2$ and let $q_1$ and $q_2$ be the orthogonal projections of $p_1$ and $p_2$ onto it, respectively. Similarly, take the outer common tangent of the disks $D_1'$ and $D_2'$ lying on the same side of the line $p_1p_2$ as the segment $[q_1,q_2]$ and define $q_1'$ and $q_2'$ as the orthogonal projections of $p_1'$ and $p_2'$ onto it, respectively.
	
	Rotate the quadrangle $p_1'q_1'q_2'p_2'$ about $p_1$ with angle $\angle q_1'p_1q_1$ in the direction taking $q_1'$ to $q_1$, and let $\tilde{q}_2'\in [q_1,q_2]$ and $\tilde{p}_2'$ be the images of the vertices $q_2'$ and $p_2'$, respectively. As $\angle \tilde{p}_2' \tilde{q}_2'q_2=\angle {p}_2' {q}_2'q_1'$ and $\angle p_2q_2\tilde{q}_2'$ are right angles and $d(\tilde{p}_2', \tilde{q}_2')=d(p_2,q_2)=r_2$, the quadrangle $\tilde{p}_2' \tilde{q}_2'q_2p_2$ is a Saccheri quadrilateral. This means that the orthogonal bisector $l$ of the side $[q_2,\tilde{q}_2']$ is a symmetry axis of the quadrangle, therefore it bisects orthogonally also the segment $[p_2,\tilde{p}_2']$. 
	
	For a pair of points $a,b\notin l$, denote by $a\sim b$ the equivalence relation that $a$ and $b$ are on the same side of $l$. As the lines $q_1p_1$, $\tilde{q}_2'\tilde{p}_2'$ and $l$ are orthogonal to the line $q_1q_2$, they are ultraparallel, therefore disjoint, consequently, $q_1\sim p_1$, $\tilde{q}_2'\sim \tilde{p}_2'$. Since $d(p_1,\tilde{p}_2')=d(p'_1,p_2')<d(p_1,p_2)$, $p_1$ is lying on the same side of the orthogonal bisector $l$ of the segment $[\tilde{p}_2',p_2]$ as the point $\tilde{p}_2'$, that is $p_1\sim \tilde{p}_2'$. However, $l$ is also the orthogonal bisector of the segment $[q_2,\tilde{q}_2']$, so  $q_1\sim p_1 \sim \tilde{p}_2'\sim \tilde{q}_2'$ implies
	\begin{equation}\label{erinto_comp}
	d(q_1',q_2')=d(q_1,\tilde{q}_2')<d(q_1,q_2).
	\end{equation}

As the disk $D_2'$ intersects the tangent $q_1q_2$ of the disk $D_1$ by \eqref{eq:nonincl}, while  $q_1'q_2'$ is an outer common tangent of the two disks, we must have
\begin{equation}\label{szog1_comp}
\angle q_1p_1p_2<\angle q_1'p_1'p_2'.
\end{equation}
\begin{figure}[t]
	\centering{\includegraphics[height=5.7cm]{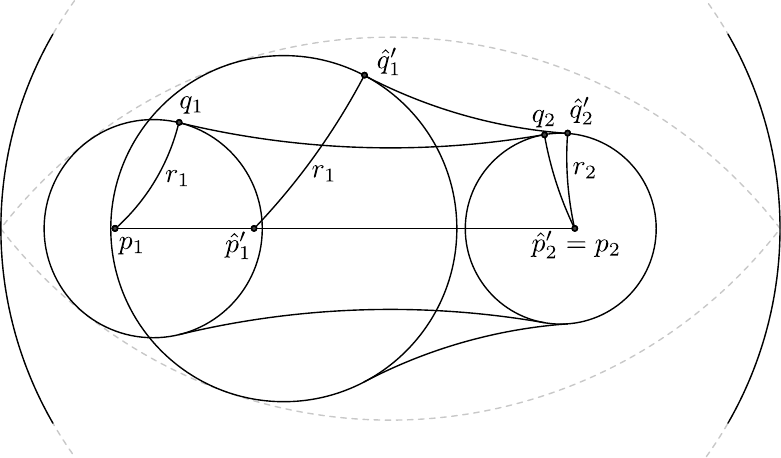}}
	\caption{Repositioning of the disks.\label{k_2_2b} }
\end{figure}
We can also prove the inequality
\begin{equation}\label{szog2_comp}
\angle q_2p_2p_1\leq \angle q_2'p_2'p_1'.
\end{equation}
For this purpose, reposition the disks of the contracted system by an isometry, which takes the points $p_1'$, $p_2'$, $q_1'$, $q_2'$ to $\hat p_1'\in [p_1,p_2]$, $\hat p_2'=p_2$, $\hat q_1'$, $\hat q_2'$, respectively, see Figure \ref{k_2_2b}. For the new arrangement, the distance of the point $\hat p_1'$ from the straight line $q_1q_2$ is at most $\max\{d(p_1,q_1),d(p_2,q_2)\}=r_1$ by the Claim. Using the fact that the disk $B(\hat p_1',r_1)$ is either tangent to the line $q_1q_2$ or crosses it, inequality \eqref{szog2_comp} follows just as inequality  \eqref{szog1_comp}.

It is known that if the sectional curvature of $\mathbb H^2$ is $-k^2$, then the perimeter of a circle of radius $r$ in $\mathbb H^2$ is  $2\pi \sigma(r)$, where $\sigma(r)=\sinh(rk)/k$. Hence we conclude by \eqref{erinto_comp}, \eqref{szog1_comp}, and \eqref{szog2_comp} that 
\begin{align*}
\per(\conv(D_1\cup D_2))&=2\big((\pi-\angle q_1p_1p_2)\sigma(r_1) +d(q_1,q_2)+(\pi-\angle q_2p_2p_1)\sigma(r_2) \big)\\
&>2\big((\pi-\angle q_1'p_1'p_2')\sigma(r_1) +d(q_1',q_2')+(\pi-\angle q_2'p_2'p_1')\sigma(r_2) \big)\\
&=\per(\conv(D_1'\cup D_2')).\qedhere
\end{align*}
\end{proof}

Let us recall the definition of the central set of a compact subset of $M$.

\begin{defin}
	Let $U\subset M$ be a compact subset. A closed disk $D\subseteq U$ (possibly with radius $0$) is said to be \emph{maximal} in $U$ if there is no closed disk $\widetilde D$ such that $D\subsetneq \widetilde D\subseteq U$.
\end{defin}

\begin{defin}
	The \emph{central set} $C_U$ of the compact subset $U\subset M$ is the set of the centers of the maximal disks in $U$.
\end{defin}

We shall consider the central set $C_U$ of the convex hull  
\[U=\conv\left(\bigcup_{i=1}^kD_i\right).\] 
\begin{figure}[b]
	\centering{\includegraphics[height=5.3cm]{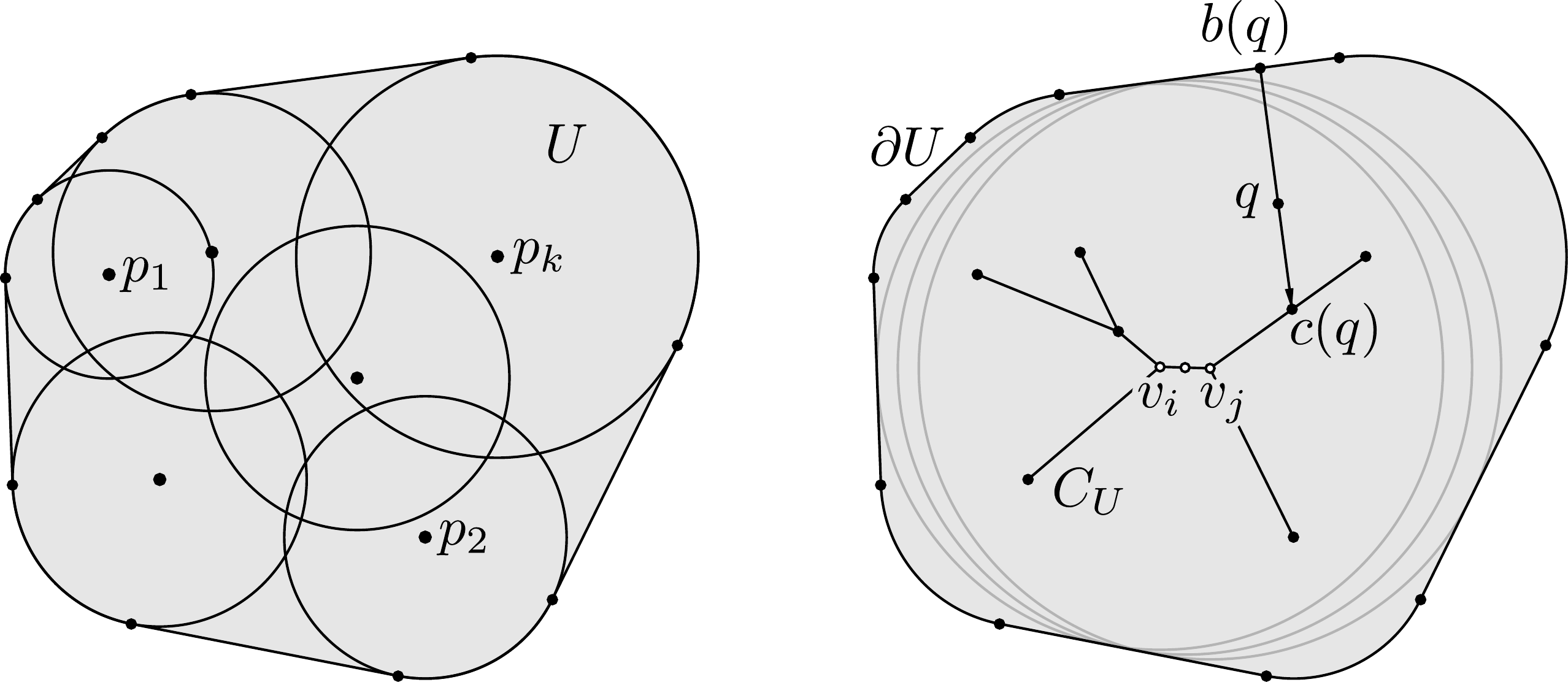}}
	\caption{The central set of $U$ and the deformation retraction of $U$ onto it.\label{retraction} }
\end{figure}
\begin{lem}\label{lemma:tree} Let $v_1,\dots, v_l$ be the centers of those maximal disks in $U$, which have at least $3$ points in common with the boundary $\partial U$ of $U$, or have radius $0$. The central set $C_U$ is a tree, the edges of which are geodesic segments connecting certain pairs of the vertices $v_1,\dots, v_l$.
\end{lem} 
\begin{proof} The boundary of $U$ is a piecewise smooth curve consisting of geodesic segments and circle arcs. If $\partial U$  is a circle, then $C_U$ is a single point, so the lemma holds. If $\partial U$ is not a circle, then every maximal disk is tangent to at least two geodesic sides of $U$.  If there exist maximal disks which are tangent to two given geodesic sides of $U$, then the centers of all such disks move along a geodesic segment, the endpoints of which belong to the set $\{v_1,\dots,v_l\}$. This segment may degenerate to a point. Thus, $C_U$ is the union of a finite number of geodesic segments of the form $[v_i,v_j]$. 
	
	For any point $q\in U\setminus C_U$, there is a unique point $b(q)\in \partial U$ for which $d(q,\partial U)=d(q,b(q))$. The geodesic entering the interior of $U$ at $b(q)$ orthogonally to $\partial U$ has a first intersection point $c(q)$ with $C_U$ after passing through $q$.  Moving each point $q\in U\setminus C_U$ simultaneously with constant speed along the geodesic segment $[q,c(q)]$ so that each point $q$ arrives in $c(q)$ after a given unit time, we obtain a deformation retraction of $U$ onto $C_U$. As $U$ is convex, hence homotopically trivial, the central set $C_U$ is also contractible, so it must be a tree, the edges of which are geodesic segments. For details of the proof, we refer to Section 4 of \cite{Gorbovickis}, for illustration, see Figure \ref{retraction}.
\end{proof}

\begin{proof}[Proof of Theorem \ref{thm:perconv}]
By Theorem \ref{Kirszbraun}, there exists a contraction $f\colon M\to M$ such that $f(p_i)=p_i'$ for $i=1,\dots,k$.

It is enough to consider the case, when all the disks $D_i$ are maximal in $U$. Indeed, assuming that this special case is true, we can prove the general case as follows. Take for every disk  $D_i$ a maximal disk $B(\tilde p_i,\tilde r_i)$ in $U$ which contains $D_i$. Then we have 
\[
U=\conv\left(\bigcup_{i=1}^k D_i\right)=\conv\left(\bigcup_{i=1}^kB(\tilde p_i,\tilde r_i)\right),
\]
and $\tilde r_i-r_i\geq d(\tilde p_i,p_i)\geq d(f(\tilde p_i),f(p_i))$, therefore $B(f(\tilde p_i),\tilde r_i)\supseteq D_i'$. Thus, applying the special case of the theorem for the maximal disks $B(\tilde p_1,\tilde r_1),\dots,B(\tilde p_k,\tilde r_k)$, and the contraction of the centers by $f$, we obtain
\begin{align*}\per\left(\conv\left(\bigcup_{i=1}^k D_i\right)\right)&=\per\left(\conv\left(\bigcup_{i=1}^kB(\tilde p_i,\tilde r_i)\right)\right)\geq \\ &\geq\per\left(\conv\left(\bigcup_{i=1}^kB(f(\tilde p_i),\tilde r_i)\right)\right)\geq \per\left(\conv\left(\bigcup_{i=1}^kD_i'\right)\right).
\end{align*}

By Lemma \ref{lemma:tree}, the central set of $U$ is a tree with vertices $v_1,\dots,v_l$. We may assume that each maximal disk $B(v_i,\rho_i)\subseteq U$ centered at a vertex $v_i$ of $C_U$ belongs to the initial system  $D_1,\dots,D_k$, because adding these disks to the initial system, and the disks $B(f(v_i),\rho_i)$ to the contracted system, inequality \eqref{ineq_perconv} becomes sharper as its left hand side remains the same while its right hand side does not decrease.

For brevity, we shall say that the initial system of disks is \emph{sharpened} if the disks $D_i$ are maximal in $U$, and $\{v_1,\dots,v_l\}\subseteq \{p_1,\dots,p_k\}$. From now on, we shall assume that the initial system of disks is sharpened, and we prove the theorem under this assumption by induction on the number $k$ of the disks. 

The initial case $k=1$ is obvious. Assume that the statement is true for any sharpened system of $k-1$ disks and its contraction, and consider a sharpened system of $k$ disks. The set $C_U$ is a tree on the vertices $v_1,\dots, v_l$, and the points $p_i$ are located in $C_U$. Subdividing each edge of $C_U$ by the points $p_i$ lying on that edge, we obtain a tree on the vertices $p_1,\dots,p_k$. This tree has a degree one vertex. We may assume without loss of generality, that $p_k$ has degree one, and its neighbor is $p_{k-1}$.

Consider the decomposition $C_U=X\cup Y$ of the central set $C_U$ into the closed subsets $X$ and $Y$, where $X$ is the subtree of $C_U$ spanned by the vertices $p_1,\dots,p_{k-1}$, and $Y$ is the segment $[p_{k-1},p_k]$. Following \cite{Gorbovickis}, for a closed subset $A\subseteq C_U$, we denote by $U_A$ the union of those maximal disks in $U$, the centers of which belong to $A$. We prove that in our case,
\begin{align}
U_Y&=\conv(D_{k-1}\cup D_k),\label{tree_decomp2}\\
U_X&=\conv\left(\bigcup_{i=1}^{k-1}D_i\right).\label{tree_decomp1}\end{align}
Indeed, by the description of the segments making up the central set of $U$, given in the proof of Lemma \ref{lemma:tree}, a maximal disk in $U$ has its center on the segment $[p_{k-1},p_k]$ if and only if it is tangent to the common outer tangent segments of the maximal disks centered at $p_{k-1}$ and $p_k$, which are the disks $D_{k-1}$ and $D_k$.  It is clear that the union of all these disks is the convex hull of the union $D_{k-1} \cup D_k$, so \eqref{tree_decomp2} holds.

Similarly, if a maximal disk $D$ in $U$ is centered in $X$, then it belongs to a segment $[p_i,p_j]$ for some $1\leq i<j\leq k-1$, and then it is tangent to both outer common tangent segments of $D_i$ and $D_j$, which implies that 
\[
D\subseteq \conv(D_i\cup D_j)\subseteq \conv\left(\bigcup_{i=1}^{k-1}D_i\right).
\]
In particular,
\[
U_X\subseteq \conv\left(\bigcup_{i=1}^{k-1}D_i\right).
\]
To show the reversed containment, choose an arbitrary point $p$ in $\conv\left(\bigcup_{i=1}^{k-1}D_i\right)$, and choose a maximal disk $B(c,r)$ in $U$ that covers $p$. If $c\in X$, or $p\in D_{k-1}$, then $p\in U_X$, and we are done. Denote the set $\conv\big(D_{k-1} \cup D_k\big)\setminus D_{k-1}$ by $W$. If $c\notin X$ and $p\notin D_{k-1}$, then $c\in Y$ and $p\in W$. We prove that this case cannot occur. The two outer common tangent lines of the disks $D_{k-1}$ and  $D_{k}$ are supporting lines of $U$, thus, each of them bounds a closed halfplane containing $U$. The intersection of these two halfplanes is a convex domain $Q\supseteq U$. The boundaries of $D_{k-1}$ and $W$ intersect in an arc, which cuts $Q$ into two closed non-overlapping parts $Q_1$, $Q_2$, so that $Q_1\supseteq W$, $Q_2\supseteq D_{k-1}$. It is clear from the construction that $W\cup Q_2\supseteq U$.  If for an index $1\leq i \leq k-2$, the disk $D_i$ had a point in the domain $W$, then it would either be contained in $W$, or it would intersect the boundary arc between $W$ and $D_{k-1}$, in which case, the difference $D_i\setminus W$ would be covered by $D_{k-1}$. In both cases, we would have $D_i\subseteq U_Y$. However, this is not possible because all maximal disks in $U_Y$ are centered at a point in $Y$, but $p_i\in Y$ would contradict the assumption that the neighbor of $p_k$ is $p_{k-1}$. The contradiction shows that the convex domain $Q_2$ contains all the disks $D_1,\dots,D_{k-1}$, and therefore, it contains also their convex hull. But then the convex hull cannot have any point in $W$, as we claimed. This completes the proof of \eqref{tree_decomp1}.

From general results of I.~Gorbovickis, \cite[Lemmata 4.14, 4.15]{Gorbovickis} we also have
\begin{align}
U_X\cap U_Y&=U_{X\cap Y}=D_{k-1},\label{tree_decomp3}\\
C_{U_X}&=X.\label{tree_decomp4}
\end{align}
Denote by $U_X'$ and $U_Y'$ the convex hulls
\[
U_X'=\conv\left(\bigcup_{i=1}^{k-1}D_i'\right)\quad\text{ and }\quad U_Y'=\conv(D_{k-1}'\cup D_k').
\]
As $U_X'$ and $U_Y'$ are convex, their union is star-like and therefore, the boundary of the union $U_X'\cup U_Y'$ is a simple closed curve. The perimeter of the convex hull of a simple closed curve is at most as long as the curve, assuming that the curve is contained in a closed hemisphere when $M=\mathbb S^2$. Thus, 
\begin{equation}\label{ineq:5}\begin{aligned}
\per\left(\conv\left(\bigcup_{i=1}^{k}D_i'\right)\right)&=\per\left(\conv(U_X'\cup U_Y')\right)\\ &\leq \per(U_X'\cup U_Y')=\per(U_X')+\per(U_Y')-\per(U_X'\cap U_Y').
\end{aligned}
\end{equation}
The disks $D_1,\dots,D_{k-1}$ form a sharpened system by \eqref{tree_decomp4}. The induction hypothesis and Lemma \ref{lemma:k_2} yield that
\begin{equation}\label{ineq:6}
\per(U_X')\leq \per(U_X)\quad \text{ and }\quad \per(U_Y')\leq \per(U_Y).
\end{equation}

By equation \eqref{tree_decomp3} and  $U_X'\cap U_Y'\supseteq D_{k-1}'$, we also have 
\begin{equation}\label{ineq:7}
\per(U_X'\cap U_Y')\geq \per(U_X\cap U_Y).
\end{equation}
Inequalities \eqref{ineq:5}, \eqref{ineq:6}, and \eqref{ineq:7} provide

\begin{equation*}\begin{aligned}
\per\left(\conv\left(\bigcup_{i=1}^{k}D_i'\right)\right)&\leq \per(U_X')+\per(U_Y')-\per(U_X'\cap U_Y')\\
&\leq \per(U_X)+\per(U_Y)-\per(U_X\cap U_Y)=\per(U_X\cup U_Y)\\
&=\per\left(\conv\left(\bigcup_{i=1}^{k}D_i\right)\right),
\end{aligned}
\end{equation*}
as we wanted to prove.
\end{proof}

\section{The area of the intersection of finitely many disks \label{sec:3}}
The dual Kneser--Poulsen conjecture was proved by 
K.~Bezdek and R.~Connelly \cite{Bezdek_Connelly} in the Euclidean plane, and by I.~Gorbovickis in the sphere under some assumptions  (Theorem \ref{thm:Gorbovickis_metszet}).
In this section, we complete these results by extending the dual Kneser--Poulsen conjecture to the hyperbolic plane. Our proof works also for the Euclidean plane, so we formulate the theorem for both planes. 

\begin{thm}\label{thm:area_int}
	If $M\in\{\mathbb H^2,\mathbb E^2\}$, and the configuration $(p_1',\dots,p_k')\in M^k$ is a contraction of the configuration $(p_1,\dots,p_k)\in M^k$, then for any choice of the radii $(r_1,\dots,r_k)$, we have
	\begin{equation} \label{ineq:intersection}
	\area\left(\bigcap_{i=1}^kB(p_i,r_i)\right)\leq \area\left(\bigcap_{i=1}^kB(p_i',r_i)\right).
	\end{equation}
\end{thm}

We shall use the notations $D_i=B(p_i,r_i)$ and $D_i'=B(p_i',r_i)$ again, and we set $U=\bigcap_{i=1}^kD_i$ and $U'=\bigcap_{i=1}^kD_i'$.

The theorem is obvious if $U$ has no interior point, because in that case, the left hand side of the inequality is $0$. From now on, we shall assume that the interior of $U$ is not empty. 




The main tool of the proof will be a dual version of the central set of a compact subset of $M$, which we shall call the co-central set.

\begin{defin}
	Let $A\subset M$ be a compact subset. A closed disk $D\supseteq A$ is said to be a \emph{minimal covering disk of} $A$ if there is no closed disk $\widetilde D\subsetneq D$ such that $\widetilde D\supseteq A$.
\end{defin}

\begin{defin}
	The \emph{co-central set} $C_A^*$ of the compact subset $A\subset M$ is the set of the centers of the minimal covering disks of $A$.
\end{defin}

The following lemmata describe the co-central set $C_U^*$ of the intersection $U$ of the disks $D_i$.
\begin{lem}\label{lem:spindle}
The radius of any minimal covering disk of $U$ is at most $\rho=\max_{1\leq i\leq k}r_i$.
\end{lem}
\begin{proof}
	For two points $p,q\in M$ with $d(p,q)\leq 2\rho$, we define the \emph{$\rho$-spindle} of $p$ and $q$ as the intersection of all those disks of radius $\rho$ which contain both $p$ and $q$.
	Recall that a subset $A\subseteq M$ of diameter at most $2\rho$ is called \emph{$\rho$-hyperconvex} (or \emph{spindle convex with radius $\rho$}) if for any two points $p,q\in A$, $A$ contains the $\rho$-spindle of $p$ and $q$.
	
	As the disks $D_i$ are all $\rho$-hyperconvex, so is their intersection $U$. The boundary circle of any minimal covering disk $D$ of $U$ intersects $U$ in at least two points $p$ and $q$, and by the $\rho$-hyperconvexity of $U$, $D$ contains also the $\rho$-spindle of $p$ and $q$. This implies that the radius of $D$ is at most $\rho$.  
\end{proof}

The main consequence of the lemma is that $C_U^*$ is compact. In general, $C_A^*$ is not compact  for any compact subset $A$ of $M$. For example, the co-central sets of convex polygons in $M$ are unbounded.

Denote by $\mathfrak D^H$, $\mathfrak D^E$, $\mathfrak D^S$, and $\mathfrak D^M$ the set of all closed disks with non-empty interior and exterior in $\mathbb H^2$, $\mathbb E^2$, $\mathbb S^2$, and $M$, respectively. The condition on the interior excludes disks of radius $0$, the condition on the exterior is relevant only on the sphere and excludes the whole sphere $\mathbb S^2$ from the space of closed disks. These sets are equipped with a natural topology. Modeling the hyperbolic plane by Poincar\'e's conformal disk model, we get a conformal embedding $\mathbb H^2\hookrightarrow \mathbb E^2$. Similarly, a stereographic projection gives a conformal embedding  $\mathbb E^2\hookrightarrow \mathbb S^2$. As under these embeddings, disks in the domain space are mapped onto disks in the codomain space, these embeddings induce embeddings of the spaces of disks $\mathfrak D^H\hookrightarrow \mathfrak D^E\hookrightarrow \mathfrak D^S$. Fixing these embeddings, we shall assume that $\mathbb H^2\subset \mathbb E^2\subset \mathbb S^2$ and $\mathfrak D^H\subset\mathfrak D^E\subset \mathfrak D^S$.
	
\begin{lem}\label{lem:spindle2}
	If $U=\bigcap_{i=1}^kD_i$ is the intersection of the disks $D_i\in \mathfrak D^M$, and $U$ has non-empty interior, then every minimal covering disk $D\in \mathfrak D^S$ of $U$ belongs to $\mathfrak D^M$.
\end{lem}
\begin{proof}
Let $\partial D$ be the boundary circle of $D$. As $D$ is a minimal covering disk of $U$, $\partial D$ must intersect $U$ in at least two points, say at $p$ and $q$. If  $D$ is not contained in $M$, then $\partial D\cap M$ is line of constant geodesic curvature in $M$, and the geodesic curvature of the line is strictly less than the geodesic curvature of any circle in $M$. Actually, $\partial D\cap M$ must be a straight line if $M=\mathbb E^2$, and it is either a straight line, or a   horocycle, or an equidistant curve of a straight line if $M=\mathbb H^2$. However, any such line through $p$ and $q$ crosses the interior of the $\rho$-spindle of $p$ and $q$, which contradicts the assumption that $D$ covers $U$.
\end{proof}
\begin{lem}\label{lemma:cotree} Let $v_1,\dots, v_l$ be the centers of those minimal disks covering $U$, the boundary circles of which have at least $3$ points in common with $U$. Then the co-central set $C_U^*$ is a tree, the edges of which are geodesic segments connecting certain pairs of the centers $v_1,\dots, v_l$. If the boundary circle of the minimal covering disk centered at $v_j$ meets $U$ at $\nu<\infty$ points, then the degree of the vertex $v_j$ in the tree $C_U^*$ is at least $\nu$.
\end{lem} 
\begin{figure}[b]
	\centering{\includegraphics[height=4cm]{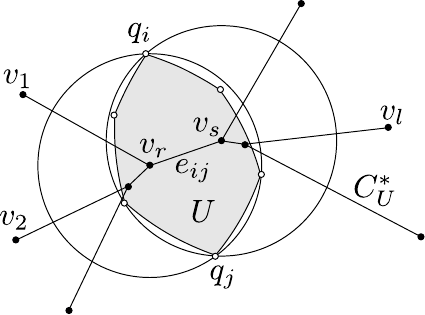}}
	\caption{Structure of the co-central set of $U$.\label{cocentral_a} }
\end{figure}
\begin{proof} If the boundary $\partial U$ of $U$ is a circle, then $C_U^*$ is a single point, so the lemma holds. If $\partial U$ is not a circle, then it is a convex curvilinear polygon, the sides of which are circle arcs meeting at some vertices $q_1,\dots,q_m$.  Then the boundary circle of every minimal covering disk of $U$ passes through at least two vertices of $\partial U$.  The set $e_{ij}$ of the centers of those minimal covering disks whose boundary circles go through the vertices $q_i\neq q_j$ is a compact convex subset of the orthogonal bisector of the segment $[q_i,q_j]$, hence it is either empty, or it is the convex hull of two of the centers $v_1,\dots,v_l$, see Figure \ref{cocentral_a}. If $e_{ij}$ is a segment, and $c$ is a point in its relative interior, then the boundary of the minimal covering disk of $U$ centered at $c$ meets $U$ only at the points $q_i$ and $q_j$. This implies that two segments of the form $e_{ij}$ and $e_{i'j'}$ corresponding to two different pairs of vertices $\{q_i,q_j\}\neq \{q_{i'},q_{j'}\}$ can intersect one another only at the endpoints of the segments. This proves that $C_U^*$ is a graph on the vertices $v_1,\dots, v_l$.

If the boundary circle of the minimal covering disk centered at $v_j$ meets $U$ at the points $q_{i_1},\dots,q_{i_{\nu}}$, listed in their cyclic order along 
the circle, then the sets $e_{i_1i_2},\dots,e_{i_{\nu-1}i_{\nu}},e_{i_{\nu}i_1}$ corresponding to the consecutive pairs of vertices are segments starting at    $v_j$, hence the degree of $v_j$ in the graph $C_U^*$ is at least $\nu$.

It remains to show that the graph $C_U^*$ is a tree. For this purpose, think of $M$ as a subset of $\mathbb S^2$ and $\mathfrak D^M$ as a subset of $\mathfrak D^S$, as explained before Lemma \ref{lem:spindle2}.
Denote by $\mathfrak C_U^*$ the set of all minimal covering disks of $U$ in $\mathfrak D^S$. Then, by Lemma \ref{lem:spindle2},  $\mathfrak C_U^* \subset\mathfrak D^M$ as well. Assigning to every point $p\in C_U^*$ the unique minimal covering disk of $U$ centered at $p$, we obtain a homeomorphism between the co-central set $C_U^*$ and the family of disks $\mathfrak C_U^*$. 

There is an involution $\iota\colon\mathfrak D^S\to \mathfrak D^S$ which maps  a spherical disk $D$ to the closure $\overline{\mathbb S^2\setminus D}$ of its complement. The subset $U$ of $M$ is a compact convex set with non-empty interior, therefore, it is a regular closed subset, which means that $U$ is equal to the closure of its interior. This implies that $\mathbb S^2\setminus U=\inter\, \overline{\mathbb S^2\setminus U}$. Applying this consequence of regularity and De Morgan's laws, we obtain that  $\iota(\mathfrak C_U^*)$ consists of maximal spherical disks lying in the union $\bigcup_{i=1}^k \iota(D_i)$. Choose a point $w\in\inter\, U$, and let $\psi\colon \mathbb S^2\setminus \{w\}\to\mathbb E^2$ be the stereographic projection from the pole $w$. This projection maps the disks $\iota(D_i)$ onto closed Euclidean disks. Assigning to a disk $D$ in $\iota(\mathfrak{C}_U^*)$ the Euclidean center of the disk $\psi(\iota(D))$, we obtain a homeomorphism between $\iota(\mathfrak{C}_U^*)$ and the central set of the union $\bigcup_{i=1}^k\psi(\iota(D_i))$. Hence, the co-central set $C_U^*$ is homeomorphic to the central set of $\bigcup_{i=1}^k\psi(\iota(D_i))$. Since $\bigcup_{i=1}^k\iota(D_i)$ is the complement of the interior of $U$ in $\mathbb S^2$, and $\inter\, U$ is contractible, results of I.~Gorbovickis \cite{Gorbovickis} yield that the central set of the union $\bigcup_{i=1}^k\psi(\iota(D_i))$ is homeomorphic to a tree graph.
\end{proof}

\begin{proof}[Proof of Theorem \ref{thm:area_int}]
	
	First we make modifications to the initial system of disks, analogous to those made in the proof of Theorem \ref{thm:perconv}.
	
	By Theorem \ref{Kirszbraun}, there exists a contraction $f\colon M\to M$ such that $f(p_i)=p_i'$ for $i=1,\dots,k$.
	
	 It is enough to prove the case when all the disks $D_i$ are minimal covering disks of $U$. Indeed, assuming that this special case is true, we can prove the general case as follows. Take for every disk  $D_i$ a minimal covering disk $B(\tilde p_i,\tilde r_i)$ of $U$ which is contained in  $D_i$. Then we have 
	\[
	U=\bigcap_{i=1}^kD_i=\bigcap_{i=1}^kB(\tilde p_i,\tilde r_i),
	\]
	and $r_i-\tilde r_i\geq d(\tilde p_i,p_i)\geq d(f(\tilde p_i),f(p_i))$, therefore $B(f(\tilde p_i),\tilde r_i)\subseteq D_i'$. Thus, applying the special case of the theorem for the minimal covering disks $B(\tilde p_1,\tilde r_1),\dots,B(\tilde p_k,\tilde r_k)$ of $U$, and the contraction of their centers by $f$, we obtain
	\begin{equation*}\area\left(\bigcap_{i=1}^kD_i\right)=\area\left(\bigcap_{i=1}^kB(\tilde p_i,\tilde r_i)\right)\leq \area\left(\bigcap_{i=1}^kB(f(\tilde p_i),\tilde r_i)\right)\leq \area\left(\bigcap_{i=1}^kD_i'\right).
	\end{equation*}
	
	By Lemma \ref{lemma:cotree}, the co-central set of $U$ is a tree with vertices $v_1,\dots,v_l$. We may assume that each minimal covering disk $B(v_i,\rho_i)\supseteq U$ centered at a vertex $v_i$ of $C_U^*$ belongs to the initial system  $\{D_j:1\leq j\leq k\}$, because adding these disks to the initial system, and the disks $B(f(v_i),\rho_i)$ to the contracted system, inequality \eqref{ineq:intersection} becomes sharper.
	
	We shall say that the initial system of disks is \emph{sharpened} if the disks $D_i$ are minimal covering disks of $U$, and $\{v_1,\dots,v_l\}\subseteq \{p_1,\dots,p_k\}$. From now on, we shall assume that the initial system of disks is sharpened, and we prove the theorem under this assumption by induction on the number $k$ of the disks. 
	
	The base case $k=1$ is obvious. The  case $k=2$ can also be proved easily, since if $r_1\geq r_2$, and we reposition the disks $D_1'$, $D_2'$ by an isometry in such a way that $p_2'$ goes to $p_2$ and $p_1'$ is moved to a point on the segment $[p_1,p_2]$, then the image of $D_1'\cap D_2'$ will cover $D_1\cap D_2$.  
	
	Assume that the statement is true for any sharpened system of $k-1$ disks and its contraction, and consider a sharpened system of $k$ disks. The set $C_U^*$ is a tree on the vertices $v_1,\dots, v_l$, and the points $p_i$ are located in $C_U^*$. Subdividing each edge of $C_U^*$ by the points $p_i$ lying in the relative interior of that edge, we obtain a tree on the vertices $p_1,\dots,p_k$. This tree has a degree one vertex. We may assume without loss of generality, that $p_k$ has degree one, and its neighbor is $p_{k-1}$, see Figure \ref{cocentral_b}.

\begin{figure}[h]
	\centering{\includegraphics[height=4cm]{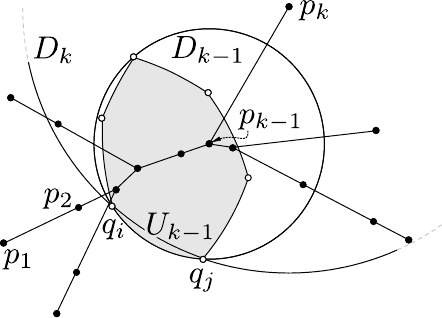}}
	\caption{The inductive step.\label{cocentral_b} }
\end{figure}
	
By Lemma \ref{lemma:cotree}, the boundary circle of $D_k$	 must intersect the boundary of $U$ in an arc, and if the endpoints of the arc are $q_i$ and $q_j$, then $p_{k-1}$ is a point of the segment $e_{ij}$ and the boundary circle of $D_{k-1}$ also passes through $q_i$ and $q_j$. From this picture, we see that the intersection $U_{k-1}=\bigcap_{i=1}^{k-1}D_i$ of the first $k-1$ disks is the disjoint union of $U$ and $D_{k-1}\setminus D_k=U_{k-1}\setminus D_k$. The co-central set of $U_{k-1}$ is just the tree $C_U^*$ from which the vertex $p_k$ is removed together with the edge $[p_{k-1},p_k]$, in particular, the system of the disks $D_1,\dots,D_{k-1}$ is also sharpened.

Denoting the intersection of $D_1',\dots,D_{k-1}'$ by $U_{k-1}'$, the induction hypothesis gives
\[
\area(U_{k-1})\leq \area(U_{k-1}').
\]
Case $k=2$ of the theorem yields
\begin{align*}
\area(D_{k-1}\setminus D_k)&=\area(D_{k-1})-\area(D_{k-1}\cap D_k)\\&\geq \area(D_{k-1}')-\area(D_{k-1}'\cap D_k') =\area(D_{k-1}'\setminus D_k').
\end{align*}
These two inequalities provide
\begin{align*}
\area(U)=\area(U_{k-1})-\area(D_{k-1}\setminus D_k)&\leq \area(U_{k-1}')-\area(D_{k-1}'\setminus D_k')\\&\leq \area(U_{k-1}')-\area(U_{k-1}'\setminus D_k')=\area(U'),
\end{align*}
as claimed.
\end{proof}

\bibliographystyle{acm}
\bibliography{KP}
\end{document}